\documentclass{amsart} 
\usepackage{amssymb,hyperref}

\newcommand{\Ann}{\text{\rm Ann}}
\newcommand{\Ass}{\text{\rm Ass}}
\renewcommand{\le}{\leqslant}
\renewcommand{\ge}{\geqslant}

\theoremstyle{plain}
\newtheorem{theorem}{Theorem}[section]

\newtheorem{lemma}[theorem]{Lemma}
\newtheorem{corollary}[theorem]{Corollary}

\theoremstyle{definition}

\newtheorem{example}[theorem]{Example}
\newtheorem*{application}{Application}

\theoremstyle{remark}

\numberwithin{equation}{theorem}

\title[Special primary decompositions in multigraded modules]
{Existence of special primary decompositions in multigraded modules}
\author{Dipankar Ghosh}
\address{Department of Mathematics, Indian Institute of Technology Bombay, Powai, Mumbai 400076, India}
\email{dipankar@math.iitb.ac.in}


\begin{document}

\begin{abstract} 
 Let $R=\bigoplus_{\underline{n} \in \mathbb{N}^t}R_{\underline{n}}$ be a commutative Noetherian $\mathbb{N}^t$-graded ring,
 and $L = \bigoplus_{\underline{n}\in\mathbb{N}^t}L_{\underline{n}}$ be a finitely generated $\mathbb{N}^t$-graded $R$-module.
 We prove that there exists a positive integer $k$ such that for any $\underline{n} \in \mathbb{N}^t$ with
 $L_{\underline{n}} \neq 0$, there exists a primary decomposition of the zero submodule
 $O_{\underline{n}}$ of $L_{\underline{n}}$ such that for any $P \in \Ass_{R_0}(L_{\underline{n}})$, the $P$-primary
 component $Q$ in that primary decomposition contains $P^k L_{\underline{n}}$. We also give an example which shows
 that not all primary decompositions of $O_{\underline{n}}$ in $L_{\underline{n}}$ have this property. As an application
 of our result, we prove that there exists a fixed positive integer $l$ such that the $0^{\rm th}$ local cohomology
 $H_I^0(L_{\underline{n}}) = \big(0 :_{L_{\underline{n}}} I^l\big)$ for all ideals $I$ of $R_0$ and for all
 $\underline{n} \in \mathbb{N}^t$.
\end{abstract}

\maketitle

\section{Introduction}\label{Introduction}
 Let $A$ be a commutative Noetherian ring with unity, and let $N\subsetneqq M$ be two finitely generated $A$-modules.
 Let $N = Q_1\cap Q_2\cap \cdots \cap Q_r$ be an irredundant and minimal primary decomposition of $N$ in $M$,
 where $Q_i$ is a $P_i$-primary submodule of $M$, i.e., $\Ass_A(M/Q_i) = \{P_i\}$ for all $i = 1,2,\ldots,r$.
 We call $Q_i$ as a $P_i$-primary component of $N$ in $M$. In this case, we have
 $\Ass_A(M/N) = \{P_1,P_2,\ldots,P_r\}$. Note that $\sqrt{\Ann_A(M/Q_i)} = P_i$, and hence there exists some
 positive integer $s_i$ such that $P_i^{s_i} M\subseteq Q_i$ for each $i = 1,2,\ldots,r$. From now onwards, by a
 primary decomposition, we always mean an irredundant and minimal primary decomposition unless explicitly stated otherwise.

 Through out this article, we denote $R = \bigoplus_{\underline{n}\in\mathbb{N}^t}R_{\underline{n}}$ as a commutative
 Noetherian $\mathbb{N}^t$-graded ring with unity (where $\mathbb{N}$ is the collection of all non-negative integers and
 $t$ is any fixed positive integer) and $L=\bigoplus_{\underline{n}\in\mathbb{N}^t}L_{\underline{n}}$ as a finitely generated
 $\mathbb{N}^t$-graded $R$-module unless explicitly stated otherwise. Set $A = R_{(0,\ldots,0)}$. Note that $A$ is a
 Noetherian ring, and each $L_{\underline{n}}$ is a finitely generated $A$-module. {\it We denote $O_{\underline{n}}$
 as the zero submodule of $L_{\underline{n}}$ for each $\underline{n}\in\mathbb{N}^t$}.
 
 For each $\underline{n} \in \mathbb{N}^t$, fix a primary decomposition 
 \begin{center}
  $(\dag)$ \hfill $O_{\underline{n}} = Q_{\underline{n},1}\cap Q_{\underline{n},2} \cap\cdots\cap Q_{\underline{n},r_{\underline{n}}}$ \hfill \;
 \end{center}
 of $O_{\underline{n}}$ in $L_{\underline{n}}$, where $Q_{\underline{n},i}$ is a $P_{\underline{n},i}$-primary
 component of $O_{\underline{n}}$ in $L_{\underline{n}}$. Then there exists a positive integer $s_i(\underline{n})$
 such that $P_{\underline{n},i}^{s_i(\underline{n})} L_{\underline{n}} \subseteq Q_{\underline{n},i}$ for each
 $\underline{n}\in\mathbb{N}^t$ and $i = 1,2,\ldots,r_{\underline{n}}$. A natural question arises that ``can we
 choose a primary decomposition $(\dag)$ of each $O_{\underline{n}}$ for which we have
 $s_i(\underline{n})$ bounded irrespective of $i$ and $\underline{n}$?".
 
 In this article, we see that there exists some primary decomposition
 $(\dag)$ of each $O_{\underline{n}}$ for which
 we can choose $s_i(\underline{n})$ in such a way that it is bounded
 (see Theorem~\ref{theorem: primary decomposition}). This confirms a conjecture by Tony J. Puthenpurakal. In an
 example, we also show that not all primary decompositions of $O_{\underline{n}}$ have this property
 (see Example~\ref{counter example}).
 \begin{application}
 Suppose $I$ is an ideal of $A$, and $M$ is an $A$-module. Recall that the $0^{\rm th}$ local cohomology
 module $H_I^0(M)$ of $M$ with respect to $I$ is defined to be
 \[
   H_I^0(M) := \big\{x\in M ~|~ I^n x = 0\mbox{ for some }n\in \mathbb{N}\big\} =
   \bigcup_{n\in \mathbb{N}}\left(0 :_M I^n\right).
 \]
 If $M$ is a Noetherian $A$-module, then it is easy to see that $H_I^0(M) = \big(0 :_M I^e\big)$ for some $e \in \mathbb{N}$,
 where $e$ depends on $M$ as well as $I$. If $L = \bigoplus_{\underline{n}\in\mathbb{N}^t}L_{\underline{n}}$ is a finitely
 generated $\mathbb{N}^t$-graded $R$-module, then as an application of our main result on primary decomposition, we prove
 that there exists a fixed positive integer $l$ such that
 $H_I^0(L_{\underline{n}}) = \left(0 :_{L_{\underline{n}}} I^l\right)$ for all ideals $I$ of $R_0$ and for all
 $\underline{n} \in \mathbb{N}^t$ (see Theorem~\ref{theorem: application of primary decomposition}).
 \end{application}

\section{Main result}\label{Main result}
 To prove our main result (Theorem~\ref{theorem: primary decomposition}), we need some preliminaries.
 We start with the following lemma.
\begin{lemma}\label{lemma: Artin-Rees}
 Let $R = \bigoplus_{\underline{n}\in\mathbb{N}^t}R_{\underline{n}}~$ be a Noetherian $\mathbb{N}^t$-graded ring, and
 $L = \bigoplus_{\underline{n} \in \mathbb{N}^t} L_{\underline{n}}$ be a finitely generated $\mathbb{N}^t$-graded $R$-module.
 Set $A=R_{(0,\ldots,0)}$. Let $J$ be an ideal of $A$. Then there exists a positive integer $k$ such that
 \[
  J^m L_{\underline{n}} \cap H_J^0(L_{\underline{n}}) = O_{\underline{n}} \quad \forall~ 
  \underline{n}\in\mathbb{N}^t\mbox{ and }\forall~ m\ge k.
 \]
\end{lemma}
\begin{proof}
 Let $I=JR$ be the ideal of $R$ generated by $J$. Since $R$ is Noetherian and $L$ is a finitely generated $R$-module,
 then by Artin-Rees lemma, there exists a positive integer $c$ such that 
 \begin{align}
 (I^m L)\cap H_I^0(L)
                &=I^{m-c} \big((I^c L)\cap H_I^0(L)\big)\quad\mbox{for all }m\ge c\nonumber\\
                &\subseteq I^{m-c} \left(H_I^0(L)\right) \quad\mbox{for all }m\ge c. \label{lemma: Artin-Rees: equation 1}
 \end{align}
 Now consider the ascending chain of submodules of $L$:
 \[\left(0:_L I\right) \subseteq \left(0:_L I^2\right) \subseteq \left(0:_L I^3\right)\subseteq \cdots.\]
 Since $L$ is a Noetherian $R$-module, there exists some $l$ such that 
 \begin{equation}\label{lemma: Artin-Rees: equation 2}
  \left(0:_L I^l\right) = \left(0:_L I^{l+1}\right) = \cdots = H_I^0(L).
 \end{equation}
 Set $k:= c+l$. Then from \eqref{lemma: Artin-Rees: equation 1} and \eqref{lemma: Artin-Rees: equation 2}, for all $m \ge k$,
 we have
 \[(I^m L)\cap H_I^0(L) \subseteq I^{m-c} \left(0 :_L I^{m-c}\right) = 0,\]
 which gives $~ J^m L_{\underline{n}}\cap H_J^0(L_{\underline{n}}) = O_{\underline{n}}~$
 for all $\underline{n} \in \mathbb{N}^t$ and for all $m\ge k$.
\end{proof}
 An immediate corollary is the following.
\begin{corollary}\label{corollary: lemma: Artin-Rees}
 Let $R=\bigoplus_{\underline{n}\in\mathbb{N}^t}R_{\underline{n}}$ be a Noetherian $\mathbb{N}^t$-graded ring, and
 $L=\bigoplus_{\underline{n}\in\mathbb{N}^t}L_{\underline{n}}$ be a finitely generated $\mathbb{N}^t$-graded $R$-module.
 Set $A=R_{(0,\ldots,0)}$. Then there exists a positive integer $k$ such that for any
 $~ P \in \bigcup_{\underline{n}\in\mathbb{N}^t}\Ass_A(L_{\underline{n}})$, we have
 \[
  P^k L_{\underline{n}} \cap H_P^0(L_{\underline{n}}) = O_{\underline{n}} \quad \forall~ \underline{n}\in\mathbb{N}^t.
 \]
\end{corollary}
\begin{proof}
 By \cite[Lemma~3.2]{We04}, we may assume that 
 \[ \bigcup_{\underline{n}\in\mathbb{N}^t} \Ass_A(L_{\underline{n}}) = \left\{P_1,P_2,\ldots,P_l\right\}.\]
 From Lemma~\ref{lemma: Artin-Rees}, for each $P_i$ ($1\le i\le l$), there exists some $k_i$ such that
 \[
   P_i^m L_{\underline{n}} \cap H_{P_i}^0(L_{\underline{n}})
   = O_{\underline{n}} \quad \forall~ \underline{n}\in\mathbb{N}^t\mbox{ and }\forall~ m\ge k_i.
 \]
 Now the corollary follows by taking $k:=\max\{k_i:1\le i\le l\}$.
\end{proof}
Let us recall the following result from \cite[Theorem~1.1]{Ya02}.
\begin{theorem}\label{theorem: compatibility}
 Let $A$ be a Noetherian ring, and $N\subsetneqq M$ be two finitely generated $A$-modules such that
 $\Ass_A(M/N) = \{P_1,\ldots,P_r\}$. Let $Q_i$ be a $P_i$-primary component of $N$ in $M$ for each $1\le i\le r$.
 Then $N = Q_1\cap \cdots \cap Q_r$, which is necessarily an irredundant and minimal primary decomposition of $N$ in $M$.
\end{theorem}
 We are now in a position to prove our main result.
\begin{theorem}\label{theorem: primary decomposition}
 Let $R=\bigoplus_{\underline{n}\in\mathbb{N}^t}R_{\underline{n}}$ be a Noetherian $\mathbb{N}^t$-graded ring, and
 $L=\bigoplus_{\underline{n}\in\mathbb{N}^t}L_{\underline{n}}$ be a finitely generated $\mathbb{N}^t$-graded $R$-module.
 Set $A=R_{(0,\ldots,0)}$.
 Then there exists a positive integer $k$ such that for each $\underline{n}\in\mathbb{N}^t$ with
 $L_{\underline{n}}\neq 0$, there exists a primary decomposition of the zero submodule $O_{\underline{n}}$ of
 $L_{\underline{n}}$:
 \[
   O_{\underline{n}}= Q_{\underline{n},1}\cap Q_{\underline{n},2}\cap \cdots\cap Q_{\underline{n},r_{\underline{n}}},
 \]
 where $Q_{\underline{n},i}$ is a $P_{\underline{n},i}$-primary component of $O_{\underline{n}}$ in
 $L_{\underline{n}}$ satisfying 
 \[
  P_{\underline{n},i}^k L_{\underline{n}} \subseteq Q_{\underline{n},i} \quad \forall~\underline{n}\in\mathbb{N}^t
  \mbox{ and }~\forall~ i = 1,2,\ldots,r_{\underline{n}}.
 \]
\end{theorem}
\begin{proof}
 Let $k$ be as in Corollary~\ref{corollary: lemma: Artin-Rees}. Fix $\underline{n} \in \mathbb{N}^t$ such that
 $L_{\underline{n}} \neq 0$. By Theorem~\ref{theorem: compatibility}, it is enough to prove that for each
 $P \in \Ass_A(L_{\underline{n}})$, there exists a $P$-primary component $Q$ of $O_{\underline{n}}$ in $L_{\underline{n}}$
 such that $P^k L_{\underline{n}} \subseteq Q$. So fix $P \in \Ass_A(L_{\underline{n}})$. From
 Corollary~\ref{corollary: lemma: Artin-Rees}, we have
 \begin{equation}\label{theorem: primary decomposition: equation 1}
  P^k L_{\underline{n}}\cap H_P^0(L_{\underline{n}}) = O_{\underline{n}}.
 \end{equation}
 
 It is easy to see that $P^k L_{\underline{n}} \neq L_{\underline{n}}$ (for instance by localization at $P$).
 If $H_P^0(L_{\underline{n}}) = L_{\underline{n}}$, then \eqref{theorem: primary decomposition: equation 1} gives
 $P^kL_{\underline{n}} = O_{\underline{n}}$, hence any $P$-primary component of $O_{\underline{n}}$ in
 $L_{\underline{n}}$ contains $P^kL_{\underline{n}}$, and hence we are through. Therefore we may as well assume
 that $H_P^0(L_{\underline{n}}) \neq L_{\underline{n}}$. Now we fix irredundant and
 minimal primary decompositions
 \begin{equation}\label{theorem: primary decomposition: equation 2}
  P^kL_{\underline{n}} = Q_1\cap\cdots\cap Q_u \quad
  \mbox{and}\quad H_P^0(L_{\underline{n}}) = Q_{u+1}\cap\cdots\cap Q_v
 \end{equation}
 of $P^kL_{\underline{n}}$ and $H_P^0(L_{\underline{n}})$ in $L_{\underline{n}}$
 respectively.
 
 We claim that $P \notin \Ass_A\left(L_{\underline{n}}/H_P^0(L_{\underline{n}})\right)$. Otherwise $P = \Ann_A(\overline{m})$
 for some $m \in L_{\underline{n}} \smallsetminus H_P^0(L_{\underline{n}})$. But then $Pm \subseteq H_P^0(L_{\underline{n}})$
 implies that $P^j(Pm) = 0$ in $L_{\underline{n}}$ for some $j \in \mathbb{N}$, i.e., $m \in H_P^0(L_{\underline{n}})$,
 which is a contradiction. Therefore
 $P \notin \Ass_A\left(L_{\underline{n}}/H_P^0(L_{\underline{n}})\right)$.
 From \eqref{theorem: primary decomposition: equation 1}, it can be noted that 
 \[
   P \in \Ass_A(L_{\underline{n}}) \subseteq   \Ass_A\left(L_{\underline{n}}/P^kL_{\underline{n}}\right) \cup
   \Ass_A\left(L_{\underline{n}}/H_P^0(L_{\underline{n}})\right).
 \]
 Since $P\notin \Ass_A\left(L_{\underline{n}}/H_P^0(L_{\underline{n}})\right)$, we have
 $P \in \Ass_A\left(L_{\underline{n}}/P^kL_{\underline{n}}\right)$,
 and hence without loss of generality, we may assume that $Q := Q_1$ is a $P$-primary submodule of $L_{\underline{n}}$.
 Now considering \eqref{theorem: primary decomposition: equation 1} and
 \eqref{theorem: primary decomposition: equation 2}, we have a primary decomposition
 \[O_{\underline{n}} = Q_1\cap \cdots \cap Q_u\cap Q_{u+1}\cap \cdots \cap Q_v,\]
 which need not be irredundant or minimal, but we get the desired $P$-primary component $Q = Q_1$ of
 $O_{\underline{n}}$ in $L_{\underline{n}}$ such that $P^k L_{\underline{n}} \subseteq Q$.
\end{proof}
 Now we give an example not all primary decompositions have $s_i(\underline{n})$ bounded.
\begin{example}\cite[page 299]{Sw97}\label{counter example}
 Let $K$ be a field. Set $A:=K[X,Y]$ as the polynomial ring in two variables $X,Y$ over $K$. Let $I=(X^2,XY)$. Set
 $R:=A[Z]$ and $L:=A/I[Z]$ with natural $\mathbb{N}$-grading structures, where $Z$ is an indeterminate. Here for each
 $n\in\mathbb{N}$, $L_n=(A/I)Z^n\cong A/I$ as $A$-modules.
 For each $O_n$ (zero submodule of $L_n$), fix the primary decomposition
 \[O_n = (x^2,xy) = (x)\cap (x^2,xy,y^{n+1}) = Q_{n,1}\cap Q_{n,2}~~\mbox{(say)},\]
 where $x,y$ are the images of $X,Y$ in $A/I$ respectively. Here $Q_{n,1}=(x)$ is a $(X)$-primary and
 $Q_{n,2}=(x^2,xy,y^{n+1})$ is a $(X,Y)$-primary submodules of $L_n$. For $n \ge 1$, the minimal $s_2(n)$ we can
 choose so that
 \[(X,Y)^{s_2(n)} L_n = (x,y)^{s_2(n)} \subseteq Q_{n,2} = (x^2,xy,y^{n+1})\]
 is $(n+1)$, which is unbounded.
\end{example}
\section{An application}\label{An application}
 As a consequence of the Theorem~\ref{theorem: primary decomposition}, we prove the following
 result (Theorem~\ref{theorem: application of primary decomposition}) which says that
 for any ideal $I$ of $A$ and for any $\underline{n}\in\mathbb{N}^t$, the $0^{\rm th}$ local cohomology module
 $H_I^0(L_{\underline{n}}) = \left(0 :_{L_{\underline{n}}} I^k\right)$, where $k$ is a fixed positive integer as
 occurring in Theorem~\ref{theorem: primary decomposition}.
 \begin{theorem}\label{theorem: application of primary decomposition}
  Let $R = \bigoplus_{\underline{n}\in\mathbb{N}^t}R_{\underline{n}}$ be a Noetherian $\mathbb{N}^t$-graded ring, and
  $L = \bigoplus_{\underline{n}\in\mathbb{N}^t}L_{\underline{n}}$ be a finitely generated $\mathbb{N}^t$-graded $R$-module.
  Set $A=R_{(0,\ldots,0)}$. Then there exists a positive integer $l$ such that for any ideal $I$ of $A$, we have
  \[ H_I^0(L_{\underline{n}}) = \left(0 :_{L_{\underline{n}}} I^l\right)\quad \forall~ \underline{n} \in \mathbb{N}^t.\]
 \end{theorem}
 \begin{proof}
  By virtue of the Theorem~\ref{theorem: primary decomposition}, there exists a positive integer $l$ such that for
  each $\underline{n} \in \mathbb{N}^t$ with $L_{\underline{n}} \neq 0$, we fix a primary decomposition of the zero
  submodule $O_{\underline{n}}$ of $L_{\underline{n}}$:
  \begin{equation}\label{theorem: application of primary decomposition: equation 1}
   O_{\underline{n}} = Q_{\underline{n},1}\cap Q_{\underline{n},2}\cap \cdots\cap Q_{\underline{n},r_{\underline{n}}},
  \end{equation}
  where $Q_{\underline{n},i}$ is a $P_{\underline{n},i}$-primary component of $O_{\underline{n}}$ in
  $L_{\underline{n}}$ satisfying 
  \begin{equation}\label{theorem: application of primary decomposition: equation 2}
   P_{\underline{n},i}^l L_{\underline{n}} \subseteq Q_{\underline{n},i} \quad \forall~\underline{n}\in\mathbb{N}^t
   \mbox{ and }~\forall~ i = 1,2,\ldots,r_{\underline{n}}.
  \end{equation}
  We claim that the theorem holds true for this $l$.
  
  Let $I$ be an arbitrary ideal of $A$. Fix an arbitrary $\underline{n} \in \mathbb{N}^t$. If
  $L_{\underline{n}} = 0$, then there is nothing to prove. So we may as well assume that $L_{\underline{n}} \neq 0$. By
  \cite[Proposition~3.13~a.]{Ei}, from \eqref{theorem: application of primary decomposition: equation 1}, we have
  \begin{equation}\label{theorem: application of primary decomposition: equation 3}
   H_I^0(L_{\underline{n}}) = \bigcap\left\{Q_{\underline{n},i} ~\middle|~ I \nsubseteq P_{\underline{n},i},
                                                                              1\le i\le r_{\underline{n}}\right\}.
  \end{equation}
  Then by using \eqref{theorem: application of primary decomposition: equation 3},
  \eqref{theorem: application of primary decomposition: equation 2} and
  \eqref{theorem: application of primary decomposition: equation 1} serially, we have
  \begin{align*}
   & \left(I^l L_{\underline{n}}\right) \cap H_I^0(L_{\underline{n}}) \\
   & \subseteq
   \left(\bigcap\left\{P_{\underline{n},i}^l L_{\underline{n}} ~\middle|~ I \subseteq P_{\underline{n},i},~ 1\le i\le r_{\underline{n}}\right\}\right)
   \bigcap \left(\bigcap\left\{Q_{\underline{n},i} ~\middle|~ I \nsubseteq P_{\underline{n},i},~ 1 \le i \le r_{\underline{n}}\right\}\right) \\
   & \subseteq
   \left(\bigcap\left\{Q_{\underline{n},i} ~\middle|~ I \subseteq P_{\underline{n},i},~ 1 \le i \le r_{\underline{n}}\right\}\right)
   \bigcap \left(\bigcap\left\{Q_{\underline{n},i} ~\middle|~ I \nsubseteq P_{\underline{n},i},~ 1 \le i \le r_{\underline{n}}\right\}\right) \\
   & = Q_{\underline{n},1}\cap Q_{\underline{n},2}\cap \cdots\cap Q_{\underline{n},r_{\underline{n}}} = O_{\underline{n}}.
  \end{align*}
  Thus for any ideal $I$ of $A$, we have
  \begin{equation}\label{theorem: application of primary decomposition: equation 4}
   \left(I^l L_{\underline{n}}\right) \cap H_I^0(L_{\underline{n}}) = 0\quad
   \forall~ \underline{n} \in \mathbb{N}^t.
  \end{equation}
  Let $x \in H_I^0(L_{\underline{n}})$. Then
  \[ 
   I^l x \subseteq \left(I^l L_{\underline{n}}\right) \cap H_I^0(L_{\underline{n}}) = 0,
  \]
  and hence $x \in \left(0 :_{L_{\underline{n}}} I^l\right)$. Thus for any ideal $I$ of $A$, we have
  \[ H_I^0(L_{\underline{n}}) = \left(0 :_{L_{\underline{n}}} I^l\right) \quad \forall~ \underline{n} \in \mathbb{N}^t,\]
  which completes the proof of the theorem.
 \end{proof}

\section*{Acknowledgements}
 I would like to express my sincere gratitude to my supervisor, Prof. Tony J. Puthenpurakal, for his generous guidance and
 valuable suggestions concerning this article. I would also like to thank NBHM, DAE, Govt. of India for providing financial
 support for this study.


\begin{thebibliography}{AAAA}
\bibitem{Ei} D. Eisenbud, {\it Commutative Algebra with a View Toward Algebraic Geometry}, Graduate Texts in Mathematics {\bf 150}, Springer-Verlag, New York, 1995.
\bibitem{Sw97} I. Swanson, {\it Powers of ideals. Primary decompositions, Artin-Rees lemma and regularity}, Math. Ann. {\bf 307} (1997), 299-313.
\bibitem{Ya02} Y. Yao, {\it Primary decomposition: compatibility, independence and linear growth}, Proc. Amer. Math. Soc. {\bf 130} (2002), 1629-1637.
\bibitem{We04} E. West, {\it Primes associated to multigraded modules}, J. Algebra {\bf 271} (2004), 427-453.
\end{thebibliography}
\end{document}